\documentclass[a4paper,12pt]{amsart}
\usepackage{amsfonts}
\usepackage{amsmath,array}
\usepackage{amssymb}
\usepackage{mathrsfs}
\usepackage[colorlinks]{hyperref}
 \usepackage{graphicx}
\usepackage{enumerate}
\usepackage[usenames]{color}
\newtheorem{thm}{Theorem}[section]
\newtheorem{cor}[thm]{Corollary}
\newtheorem{lem}[thm]{Lemma}
\newtheorem{prop}[thm]{Proposition}

\numberwithin{equation}{section}
\usepackage[english]{babel} 
\usepackage{blindtext}

\theoremstyle{definition}
\newtheorem{defn}[thm]{Definition}
\newtheorem{rem}[thm]{Remark}
\newtheorem{ex}[thm]{Example}

\begin{document}
\title{A note on $L^p$-$L^q$ boundedness of the Fourier multipliers on noncommutative spaces}
\author[Michael Ruzhansky]{Michael Ruzhansky}
\address{
 Michael Ruzhansky:
  \endgraf
 Department of Mathematics: Analysis, Logic and Discrete Mathematics,
  \endgraf
 Ghent University, Ghent,
 \endgraf
  Belgium 
  \endgraf
  and 
 \endgraf
 School of Mathematical Sciences, Queen Mary University of London, London,
 \endgraf
 UK  
 \endgraf
  {\it E-mail address} {\rm michael.ruzhansky@ugent.be}
  }

 \author[Kanat Tulenov]{Kanat Tulenov}
\address{
  Kanat Tulenov:
  \endgraf
  Institute of Mathematics and Mathematical Modeling, 050010, Almaty, 
  \endgraf
  Kazakhstan 
  \endgraf
  and 
  \endgraf
Department of Mathematics: Analysis, Logic and Discrete Mathematics
  \endgraf
 Ghent University, Ghent,
 \endgraf
  Belgium
  \endgraf
  {\it E-mail address} {\rm kanat.tulenov@ugent.be} 
  }
\begin{abstract}
In this work, we study Fourier multipliers on noncommutative spaces. In particular, we show a simple proof of $L^p$-$L^q$ estimate of Fourier multipliers on general noncommutative spaces associated with semifinite von Neumann algebras. This includes the case of Fourier multipliers on general locally compact unimodular groups.
\end{abstract}

\maketitle
\section{Introduction}

 The $L^p$-$L^q$ boundedness of Fourier multipliers in $\mathbb{R}^d$ was obtained in 1960 by L. H\"ormander. More precisely, if $1 < p \leq  2\leq q <\infty,$ then the classical H\"ormander's result obtained in \cite[Theorem 1.11]{Hor} says that the Fourier multiplier $g(D):\mathcal{S}(\mathbb{R}^d)\to\mathcal{S}(\mathbb{R}^d)$ with the symbol $g:\mathbb{R}^d \to \mathbb{C}$ defined by $\widehat{g(D)(f)}=g \cdot \widehat{f}, \quad f\in\mathcal{S}(\mathbb{R}^d),$ has a bounded extension from $L^p(\mathbb{R}^{d})$ to $L^q(\mathbb{R}^{d}),$ if its symbol $g$ satisfies the condition 
$$\sup\limits_{\lambda>0}\lambda\left(\int\limits_{x\in\mathbb{R}^d \atop  |g(x)|\geq\lambda}dx\right)^{\frac{1}{p}-\frac{1}{q}}<\infty, \quad 1 < p \leq  2\leq q <\infty,$$
where $\mathcal{S}(\mathbb{R}^d)$ is the Schwartz class and
$\widehat{f}$ is the Fourier transform
$$\widehat{f}(t)=(2\pi)^{-d/2}\int_{\mathbb{R}^d}f(s)e^{-i(t,s)}ds, \quad t\in \mathbb{R}^d,$$ and $(t,s)$  the standard inner product in $\mathbb{R}^d.$  
In this work, we study the $L^p$-$L^q$ boundedness of the Fourier multipliers in a more general setting, namely, noncommutative spaces associated with semifinite von Neumann algebras. In other words, we obtain an analogue of the above mentioned H\"ormander Fourier multiplier theorem on the noncommutative spaces when $1 < p \leq  2\leq q <\infty.$ 
 Since our main focus is only $L^p$-$L^q$ boundedness, for the $L^p$-$L^p$ boundedness of Fourier multipliers we refer the reader to the papers \cite{Hor, KM, Mikh, Mar, Stein}. For the $L^p$-$L^q$ boundedness of the Fourier multipliers in defferent spaces, we refer the reader to \cite{AR, AMR, ARN, CK, Cow1993, Cow1974, K2, K3, K4, Mc, M-Nur, NT1, NT2, RV, RST1, RST2, Z}, and references therein. The idea of the proof of the $L^p$-$L^q$ boundedness of the Fourier multipliers in almost all above listed papers including  H\"ormander's paper \cite{Hor} himself relies
upon the Paley inequality and Hausdorff-Young-Paley inequality for the Fourier transform. Similar ideas were also used in \cite{AR} to extend H\"ormander's result to more general locally compact groups and group von Neuman algebras. However, we have found recently that there is a gap in the proof of \cite[Theorem 3.1]{AR}. Because of that, one of the main results \cite[Theorem 5.1]{AR} is not proved in full generality. In this paper, we provide an alternative approach for the $L^p$-$L^q$ boundedness of the Fourier multipliers which rectifies the gap contained in the proof of \cite[Theorem 3.1]{AR} and show that both \cite[Theorem 3.1]{AR} and \cite[Theorem 5.1]{AR} remain true in even more general settings. This idea was used very recently in \cite{Z}, where the author obtained similar results in the context of locally compact quantum groups. The idea is very simple which does not require Paley-type inequalities, which will be used in the proof of the main result Theorem \ref{Hormander-thm}. Instead, we prove the Paley inequality with this simple approach in our setting (see Theorem \ref{Paley-ineq}).
In order to obtain main results in general semifinite von Neumann algebras, we introduced a general notion of a Fourier structure. The Fourier
structure introduced in our paper is indeed a novel concept developed to address
specific challenges in the study of Fourier multipliers within general semifinite
von Neumann algebras. Unlike in classical settings, Fourier multipliers do not
exist in general von Neumann algebras. To overcome this limitation, we constructed the Fourier structure to enable the analysis of Fourier multipliers in this broader context.
Indeed, this Fourier
structure was first studied in an unpublished draft by the first author and Rauan Akylzhanov in 2018. Our motivation for introducing this structure is twofold:

To extend the study of Fourier multipliers to general semifinite von Neumann algebras, thereby addressing a gap in the existing mathematical framework.

To ensure that this structure naturally encompasses locally compact quantum groups, which generalize the classical theory of locally compact groups. This alignment with quantum groups highlights the versatility and relevance of the Fourier structure in contemporary mathematical analysis. 

Moreover, to demonstrate its applicability we showed two concrete examples, which exhibit properties that align with this new framework.
As this concept is introduced for the first time in this paper, there are no prior references directly defining or studying this structure. However, the underlying ideas draw inspiration from the general theory of locally compact groups (see \cite{Cas}) and von Neumann algebras. Also, our results cover the main results in \cite[Theorem 5.1]{AR} and \cite[Theorem 1.3]{Z} on locally compact groups and quantum groups, respectively.

\bigskip
\section{Preliminaries}

\subsection{Classical $L^p$ and $L^{p,q}$ spaces}
For $1\leq  p\leq\infty$ and $\mathbb{R}_{+}:=(0,\infty),$ as usual $L^p(\mathbb{R}_{+})$ is the $L^p$-spaces of pointwise almost-everywhere equivalence classes of $p$-integrable functions and $L^{\infty}(\mathbb{R}_{+})$ is the space of essentially bounded functions on the half line $\mathbb{R}_{+}.$ Let $0<  p,q\leq\infty.$ Then the Lorentz space on $\mathbb{R}_{+}$ is the space of complex-valued measurable functions $f$ on $\mathbb{R}_{+}$ such that the following quasinorm is finite 
$$\|f\|_{L^{p,q}(\mathbb{R}_{+})}=\left\{ \begin{array}{rcl}
         \left(\int\limits_{\mathbb{R}_{+}}\left(t^{\frac{1}{p}}f^*(t)\right)^q\frac{dt}{t}\right)^{\frac{1}{q}}, & \mbox{for}
         & q<\infty, \\ \sup\limits_{t>0}t^\frac{1}{p}f^*(t),\;\;\;\;\;\;\;\;\;\;\;\; & \mbox{for} & q=\infty,  
                \end{array}\right. 
$$
where $f^*$ is the decreasing rearrangement of the function $f.$ For more details about these spaces, we refer the reader e.g. to \cite{G2008}.

\subsection{Von Neuman algebras and $\tau$-measurable operators}
Let $\mathcal{H}$ be a complex Hilbert space and $B(\mathcal{H})$ be the algebra of all linear bounded operators on $\mathcal{H}.$ Let us denote by $\mathcal{M}\subset B(\mathcal{H})$ a semifinite  von Neumann algebra equipped with a faithful normal semifinite trace $\tau.$ The pair $(\mathcal{M},\tau)$ is called a \textit{non-commutative measure space}. A closed and densely defined operator $x$ on $\mathcal{H}$ with the domain $\mathrm{dom}(x)$ is said to be \textit{affiliated} with $\mathcal{M}$ if $u^{\ast}xu=x$ for each unitary operator $u$ in the commutant $\mathcal{M}'$ of $\mathcal{M}$. A closed and densely defined operator $x$ is called \textit{$\tau$-measurable} if $x$ is affiliated with $\mathcal{M}$ and for every $\varepsilon>0$ there exists a projection $p\in \mathcal{M}$ such that $p(H)\subset \mathrm{dom}(x)$ and $\tau(\mathbf{1}-p)<\varepsilon.$
The set of all $\tau$-measurable operators will be denoted by $L^0(\mathcal{M})$. Given a self-adjoint operator $x\in L^0(\mathcal{M})$ and a Borel set $\mathbb{B}\subset\mathbb{R}$, we denote by  $e^{|x|}(\mathbb{B})$ its spectral projection.
For any $x\in L^0(\mathcal{M})$ define the \textit{distribution function} by
$$d(s;|x|)=\tau(e^{|x|}(s,\infty)), \quad -\infty<s<\infty.$$
It is clear that the function $d(\cdot;|x|):[0,\infty)\rightarrow[0,\infty]$ is non-increasing and the normality of the trace implies that $d(s;|x|)$ is right-continuous.
Then the \textit{generalised singular value function} of $x$ is defined by
$$\mu(t;x)=\inf\left\{s>0: d(s;|x|)\leq t\right\}, \quad t>0.$$
The function $t\mapsto\mu(t;x)$ is decreasing and right-continuous and is the right inverse of the distribution function of $x.$ 
Moreover, we have for $x,y \in L^0(\mathcal{M},\tau)$ (see for instance \cite[Corollary 2.3.16 (b)]{LSZ}, \cite{FK})
\begin{equation}\label{decreas-proper}
\mu(t+s,xy)\leq\mu(t,x)\cdot\mu(s,y),\quad t,s>0.
\end{equation} For more discussion on generalised singular value functions, we refer the reader to 
\cite{DPS, FK, LSZ}.

\subsection{Noncommutative $L^p$ and $L^{p,q}$ spaces associated with a semifinite von Neuman algebra}

For $0<p<\infty$, define the noncommutative $L^p$-space
$$L^{p}(\mathcal{M})=\{x\in L^{0}(\mathcal{M}):\tau\left(|x|^{p}\right)<\infty\},$$
where,
$$\|x\|_{L^{p}(\mathcal{M})}=\left(\tau\left(|x|^{p}\right)\right)^{\frac1p}, \,\;x\in L^{p}(\mathcal{M}).$$
Then, the space $\left(L^{p}(\mathcal{M}), \|\cdot\|_{L^{p}(\mathcal{M})}\right)$ is a quasi-Banach (Banach for $p\geq 1$) space. This space is the
noncommutative $L^{p}$-space associated with $(\mathcal{M},\tau)$, denoted by $L^{p}(\mathcal{M},\tau)$ or simply
by $L^{p}(\mathcal{M})$. As usual, we set $L^{\infty}(\mathcal{M},\tau)=\mathcal{M}$ equipped with the operator norm. Moreover, we have $$\|x\|_{L^{\infty}(\mathcal{M})}=\|\mu(x)\|_{L^{\infty}(0,\infty)}=\mu(0,x),$$ $x\in L^{\infty}(\mathcal{M}).$

\begin{defn}Let $0 < p, q\leq \infty$ and let $\mathcal{M}$ be a semifinite von Neumann algebra. Then, define the non-commutative Lorentz $L^{p,q}(\mathcal{M})$ space by
$$L^{p,q}(\mathcal{M}):=\{x\in L^{0}(\mathcal{M}):\|x\|_{L^{p,q}(\mathcal{M})}<\infty\},$$
where
$$
\|x\|_{L^{p,q}(\mathcal{M})}:=\left(\int_{
\mathbb{R}_+}\left(t^\frac{1}{p}\mu(t;x)\right)^{q}\frac{dt}{t}\right)^{\frac1q},\,\ \text{for} \,\ q<\infty,
$$
and 
$$\|x\|_{L^{p,q}(\mathcal{M})}:=\sup_{t>0}t^{\frac{1}{p}}\mu(t,x), \,\ \text{for} \,\ q=\infty.$$
In other words, we have 
\begin{equation}\label{NC Lorentz space norm II}\|x\|_{L^{p,q}(\mathcal{M})}=\|\mu(x)\|_{L^{p,q}(\mathbb{R}_+)}.
\end{equation}
In particular,
$L^{p,p}(\mathcal{M})=L^p(\mathcal{M})$ isometrically for $1<p<\infty.$
\end{defn}
For more details on noncommutative spaces associated with von Neuman algebras, we refer the reader to \cite{DPS, LSZ, PXu}.
The following are well known results for the classical Lorentz $L^{p,q}$ spaces which play key role for our investigation.
\begin{prop}\label{0.1}

\begin{enumerate}[\rm(i)]

   \item \cite[Proposition 1.4.10, p. 49]{G2008} Let $ 0 < p \leq \infty $ and $0 < q < r \leq\infty.$ Then there exists a constant $C_{p,q,r}>0$ (which depends on p, q, and r) such that
\begin{eqnarray}\label{Lorentz-embedding-1}
\|f\|_{L^{p,r}(\mathbb{R}_{+})}\leq C_{p,q,r}\|f\|_{L^{p,q}(\mathbb{R}_{+})}, \quad f\in L^{p,q}(\mathbb{R}_{+}).    
\end{eqnarray} 
\item \cite[Exercises 1.4.19, p. 73]{G2008} Let $0<p,q,r\leq \infty,$ $0<s_1,s_2\leq\infty,$    $\frac{1}{p}+\frac{1}{q}=\frac{1}{r}$ and $\frac{1}{s_1}+\frac{1}{s_2}=\frac{1}{s}.$ Then for any $f\in L^{p,s_1}(\mathbb{R}_{+})$ and $g\in L^{q,s_2}(\mathbb{R}_{+})$ there exists a constant $C_{p,q,s_1,s_2}>0$ (which depends on p, q, $s_1,$ and  $s_2$) such that,
\begin{equation}\label{Lorentz-embedding-2}
\|fg\|_{L^{r,s}(\mathbb{R}_{+})}\leq C_{p,q,s_1,s_2}\|f\|_{L^{p,s_1}(\mathbb{R}_{+})}\|g\|_{L^{q,s_2}(\mathbb{R}_{+})}.          \end{equation} 
\end{enumerate}
    \end{prop}
Note that the noncommutative extensions of the above results for noncommutative Lorentz spaces associated with a semifinite  von Neuman algebra can be found from \cite{DPS} and \cite{PXu} (see also  \cite{Z}).
\begin{prop}\label{prop-0.2}

\begin{enumerate}[\rm(i)]

   \item \cite[Theorem 6.7.12 (ii), p. 402]{DPS}, \cite[Lemma 2.1 (ii)]{Z} If $ 1\leq p < \infty $ and $1\leq q \leq r \leq\infty,$ then $L^{p,q}(\mathcal{M}) \subset L^{p,r}(\mathcal{M})$ and there exists a constant $C_{p,q,r}>0$ (which depends on p, q, and r) such that
\begin{eqnarray}\label{Lorentz-embedding-1-oper}
\|x\|_{L^{p,r}(\mathcal{M})}\leq C_{p,q,r}\|x\|_{L^{p,q}(\mathcal{M})},\quad x\in L^{p,q}(\mathcal{M}).     
\end{eqnarray} 
\item \cite[Lemma 2.2]{Z} Let $0<p_0,p_1<\infty$ and $0<q<\infty$ such that   $\frac{1}{p}=\frac{1}{p_0}+\frac{1}{p_1}.$ Then we have
\begin{eqnarray}\label{Lorentz-embedding-2-oper}
\|xy\|_{L^{p,q}(\mathcal{M})}\leq 2^{\frac{1}{p}}\|x\|_{L^{p_0,\infty}(\mathcal{M})}\|y\|_{L^{p_1,q}(\mathcal{M})}, x\in L^{p_0,\infty}(\mathcal{M}), y\in L^{p_1,q}(\mathcal{M}).   
\end{eqnarray} 
\end{enumerate}
    \end{prop}
\subsection{Fourier structure on semifinite von Neumann algebras}
Let $\mathcal{H}$ be a separable Hilbert space and let $\mathcal{M}$ and $\widehat{\mathcal{M}}$ be semifinite von Neumann algebras with normal semifinite faithful tracial states $\tau$ and $\widehat{\tau},$ respectively.
 We denote by $\mathcal{M}_{*}$ the predual of $\mathcal{M}$, i.e. $\left(\mathcal{M}_{*}\right)^{*} \cong \mathcal{M}$. With each trace $\tau,$ we can associate the Hilbert space $L^{2}(\mathcal{M})$ via the GNS construction.

\begin{defn}\label{Fourier-structure}(Fourier structure). We shall say that $\mathcal{M} \subset B(\mathcal{H})$ and $\widehat{\mathcal{M}} \subset B(\mathcal{H})$ are endowed with a Fourier structure if:
\begin{enumerate}[{\rm (F1)}]
\item  There is a linear map $\mathcal{F}_{1}: \mathcal{M}_{*} \rightarrow \widehat{\mathcal{M}}$ such that

\[
\begin{array}{r}
\left\|\mathcal{F}_{1}[x]\right\|_{\widehat{\mathcal{M}}} \leq\|x\|_{\mathcal{M}_{*}}, \quad x \in \mathcal{M}_{*} \cap \mathcal{M}, 
\quad \|x\|_{\mathcal{M}} \leq \widehat{\tau}\left(\left|\mathcal{F}_{1}[x]\right|\right).
\end{array}
\]

\item  There exists a bounded linear map $\mathcal{F}_{2}: L^{2}(\mathcal{M}) \rightarrow L^{2}(\widehat{\mathcal{M}})$, and

\begin{equation*}
\widehat{\tau}\left(\mathcal{F}_{2}[x] \mathcal{F}_{2}[y]^{*}\right)=\tau\left(x y^{*}\right).
\end{equation*}

\item  The linear maps $\mathcal{F}_{1}$ and $\mathcal{F}_{2}$ agree on the intersection $\mathcal{M}_{*} \cap \mathcal{M}$, i.e.

\begin{equation*}
\mathcal{F}_{1}(x)=\mathcal{F}_{2}(x), \quad x \in \mathcal{M}_{*} \cap \mathcal{M}.
\end{equation*}
\end{enumerate}
\item 
\end{defn} 

\begin{defn}\label{Fourier-transform}We shall say that $\mathcal{M}$ admits a Fourier structure if there exists a triple $\left(\widehat{\mathcal{M}}, \mathcal{F}_{1}, \mathcal{F}_{2}\right)$ satisfying conditions (F1)-(F3). In this case, on the intersection $\mathcal{M}_{*} \cap \mathcal{M}$ the map $\mathcal{F}_1=\mathcal{F}_{2}$ will be denoted by $\mathcal{F}.$
If $\mathcal{M}$ admits a Fourier structure, then a linear map 
$$\mathcal{F}: \mathcal{M}_{*} \cap \mathcal{M} \rightarrow \widehat{\mathcal{M}}, $$
$$\mathcal{F}:[x]\to \mathcal{F}[x]$$
is called a {\it Fourier transform} on $\mathcal{M}.$
\end{defn}

\begin{ex} Let $\mathcal{H}$ be a separable Hilbert space and let  $\mathcal{M}$ be a semifinite von Neumann algebra. Let us fix a linear self-adjoint operator $\mathcal{D}$ affiliated with $\mathcal{M}$. Kadison and Singer showed that there is unique maximal abelian algebra $\langle\mathcal{D}\rangle$ generated by $\mathcal{D}$. The spectral expansion in generalized eigenfunctions associated with arbitrary self-adjoint operators in separable Hilbert spaces has been obtained in \cite{LT16}.

Let $\left(W,\left(\mathcal{H}^{\lambda}\right)_{\lambda \in \operatorname{sp}(\mathcal{D})}\right)$ be the direct integral decomposition associated with $\operatorname{sp}(\mathcal{D})$ and where $W: L^{2}(\mathcal{M}) \rightarrow \bigoplus \int\limits_{\operatorname{sp}(\mathcal{D})} \mathcal{H}^{\lambda} d \mu(\lambda)$. Here, $\mu(\lambda)$ is a spectral measure for $\mathcal{D}$. It is a standard fact that spectral measures exist.

Let $\mathcal{S}$ be an algebraic subspace of $\mathcal{H}$ and let $\mathcal{S}^{*}$ be the algebraic dual of $\mathcal{S}$. An element $e^{\lambda} \in \mathcal{S}^{*}$ is called a generalized eigenfunction of $\mathcal{D}$ corresponding to the eigenvalue $\lambda \in \operatorname{sp}(\mathcal{D})$ if

$$
e^{\lambda}(\mathcal{D} u)=\lambda e^{\lambda}(u), \quad u \in \mathcal{S} \cap \operatorname{Dom}(\mathcal{D}).
$$
Let us define the von Neumann algebra $\widehat{\mathcal{M}}$ as the commutant of $\langle\mathcal{D}\rangle$, i.e.

\begin{equation}\label{2.8}
\widehat{\mathcal{M}}:=[\langle\mathcal{D}\rangle]^{!}. 
\end{equation}
Then the Fourier map $\mathcal{F}_{1}: \mathcal{M}_{*} \rightarrow \widehat{\mathcal{M}}$ is given by
$$
\mathcal{F}_{1}: \mathcal{M}_{*} \ni x \mapsto \widehat{x}=\bigoplus \int_{\operatorname{sp}(\mathcal{D})} \widehat{x}(\lambda) d \mu(\lambda) \in \widehat{\mathcal{M}},
$$
where $\widehat{x}(\lambda)=\operatorname{diag}\left(e_{1}^{\lambda}\left(x^{\lambda}\right), e_{2}^{\lambda}\left(x^{\lambda}\right),\ldots\right): \mathcal{H}^{\lambda} \rightarrow \mathcal{H}^{\lambda}$. Here, the set $\left\{e_{k}^{\lambda}\right\}_{k=1}^{\infty}$ is a Hilbert space basis in $\mathcal{H}^{\lambda}$ and

\begin{equation*}
\|x\|_{\mathcal{H}}^{2}=\int_{\operatorname{sp}(\mathcal{D})}\left\|x^{\lambda}\right\|_{\mathcal{H}^{\lambda}}^{2} d \mu(\lambda), \quad \text { for every } x \in \mathcal{H}.
\end{equation*}
By \cite[Theorem 2, p. 6]{LT16}, the elements $e_{k}^{\lambda}$ are generalised eigenfunctions of $\mathcal{D}$ corresponding to $\lambda \in \operatorname{Sp}(\mathcal{D})$, i.e. $\left(D e_{k}^{\lambda}, u\right)_{\mathcal{H}}=\lambda\left(e_{k}^{\lambda}, u\right)_{\mathcal{H}},  u \in \mathcal{S} \cap \operatorname{Dom}(\mathcal{D})$. It can be checked that the conditions in Definition \ref{Fourier-structure} can be satisfied. It is sufficient to verify (F1) for each factor in the direct integral decomposition. Indeed, for each $\lambda \in \operatorname{Sp}(\mathcal{D})$, without loss of generality, we can assume that $\sup\limits_{k}\left\|e_{k}^{\lambda}\right\| \leq 1$. Then we have

$$
\|\widehat{x}(\lambda)\|_{B\left(\mathcal{H}^{\lambda}\right)}=\sup _{k}\left|e_{k}^{\lambda}(x)\right| \leq\|x\|_{\mathcal{M}_{*}},
$$
for each $x \in L^2(\mathcal{M}) \cap \mathcal{M}_{*}$.

The Plancherel decomposition in (F2) holds true due to the spectral expansion (see \cite[Formula (PF)]{LT16} for more details). Condition (F3) is used just for the well-definedness of the Fourier transform.
\end{ex}
Another prototypical example is a locally compact quantum group.

\begin{ex} Let $\left(\mathbb{G}, \Delta, \varphi_{L}, \varphi_{R}\right)$ be a locally compact quantum group where $\mathbb{G}$ is a $C^{*}$-algebra, $\Delta$ is the coproduct map and $\varphi_{L}, \varphi_{R}$ are left and right Haar weights, respectively. By the Kusterman-Vaes-Pontryagin duality (see \cite{Cas},\cite{KV}), there exists the dual quantum group $\left(\widehat{\mathbb{G}}, \widehat{\Delta}, \widehat{\varphi}_{L}, \widehat{\varphi}_{R}\right)$. Assume that left and right Haar weights $\varphi_{L}, \varphi_{R}$ and their duals $ \widehat{\varphi}_{L}, \widehat{\varphi}_{R}$ are tracial.  Let us denote by $L^{\infty}(\mathbb{G})$ the corresponding von Neumann algebra. We take
$$
\mathcal{M}=L^{\infty}(\mathbb{G}), \quad \widehat{\mathcal{M}} =L^{\infty}(\widehat{\mathbb{G}}), $$
$$
\omega =\varphi_{L}, \quad \widehat{\omega}=\widehat{\varphi}_{R}.
$$
Then $\mathcal{M}$ is a von Neumann algebra and $\omega$ is a weight on $\mathcal{M}$. We denote by $\mathcal{M}_{*}$ the predual of $\mathcal{M},$ which is identified with $L^1(\mathcal{M})$ via the map
$j(x)=\omega(x\cdot).$
Then conditions (F1)-(F3) in Definition \ref{Fourier-structure} are satisfied and for detailed discussions we refer the reader to \cite[Theorem 5.2]{Cas} and \cite[Section 3.2]{Z}.
Indeed, let us introduce a map $\lambda$
\begin{equation}\label{2.13}
\lambda: \mathcal{M}_{*} \ni \omega \mapsto(\omega \otimes 1) W \in B(\mathcal{H}),
\end{equation}
where $W$ is the unitary operator implementing the coproduct $\Delta$, i.e.
$$
\Delta(a)=W(1 \otimes a) W^{*}.
$$
The dual group $\widehat{\mathbb{G}}$ is given by
$$
\widehat{\mathbb{G}}=\left\{(x \otimes 1)(W): x \in L^{1}(\mathbb{G})\right\}^{\text {strong topology closure }}.
$$
The $L^{1}(\mathbb{G}) \to L^{\infty}(\mathbb{G})$ Fourier map $\mathcal{F}_1$ is given by
\begin{equation*}
L^{1}(\mathbb{G}) \in x \mapsto \widehat{x}=\lambda(x)=(x \otimes 1) W \in L^{\infty}(\widehat{\mathbb{G}}).
\end{equation*}
The $L^{2}$-Fourier theory $\mathcal{F}_2$ is given by
$$
L^{2}(\mathbb{G}) \ni x \mapsto \lambda(x \omega)=(x \omega \otimes 1) W \in L^{2}(\widehat{\mathbb{G}}).
$$
The Plancherel identity takes the form (see \cite[Theorem 5.2]{Cas}, \cite[Theorem 2.3.11]{Tim08}, \cite{Z})
\begin{equation*}
\|\lambda(x w)\|_{L^{2}(\widehat{\mathbb{G}})}=\|x\|_{L^{2}(\mathbb{G})},
\end{equation*}
where $\lambda$ is given by \eqref{2.13}.
\end{ex}

\begin{defn}\label{Fourier-multp}  Let $\mathcal{M}$ admit a Fourier structure $(\widehat{\mathcal{M}},\mathcal{F}_1,\mathcal{F}_2)$ and let $\widehat{\mathcal{M}}$ admit a Fourier structure $(\mathcal{M}, \widehat{\mathcal{F}}_1, \widehat{\mathcal{F}}_2)$. We assume that 
    \begin{equation*}
    (\mathcal{F}\cdot \widehat{\mathcal{F}})[y] = y, \quad y\in L^2(\widehat{\mathcal{M}}),
    \end{equation*}
    \begin{equation*}
    (\widehat{\mathcal{F}}\cdot \mathcal{F})[x] = x, \quad x \in L^2(\mathcal{M}).
    \end{equation*}
    Let $\sigma\in L^0(\widehat{\mathcal{M}}).$ Then a linear map $$A_{\sigma}:\mathcal{M}_{*}\cap\mathcal{M} \ni x\mapsto\widehat{\mathcal{F}}\cdot\sigma\cdot\mathcal{F}[x]\in L^{0}(\mathcal{M})$$ is called a {\it Fourier multiplier} on $\mathcal{M},$ with the symbol $\sigma.$
    
    Moreover, if this map can be extended to a bounded map from $L^p(\mathcal{M})$ to $L^p(\mathcal{M}),$ then we call it the {\it $L^p$-$L^q$ Fourier multiplier} on $\mathcal{M}.$ 
\end{defn}

\begin{rem}The question how to define the Fourier multiplier on $\mathcal{M}$
comes from the paper \cite{AR}. Indeed, if $G$ is a locally compact separable unimodular group and $VN(G)$ is a group von Neumann algebra, then we can see that
$A$ is a left Fourier multiplier on $G$ if and only if $A$ is affiliated with the right von Neumann algebra $VN_R(G)$ and is $\tau$-measurable. Operators affiliated with the right von Neumann algebra $VN_R(G)$ are precisely those $A$ that are left-invariant on $G,$ that is,
$$A\pi_L(g)=\pi_L(g) A, \quad \text{for all}\quad g\in G,$$
where $\pi_L(g)$ is the left action of $G$ on $L^2(G).$
In other words, left Fourier multipliers on $G$ are
precisely the left-invariant operators that are $\tau$-measurable.
For more details, see \cite[Definition 2.16]{AR} and \cite[Remark 2.17]{AR}.
\end{rem}

\section{$L^p$ -$L^q$ boundedness of the Fourier multipliers and Paley type inequality}
 We start with the Hausdorff-Young inequality.
\begin{lem}(Hausdorff-Young inequality)\label{L.5.2} Let $1\leq p\leq2$ with $\frac{1}{p}+\frac{1}{p'}=1.$ Then we have
 \begin{eqnarray}\label{R-H-Y_ineq}
\|\mathcal{F}[x]\|_{L^{p',p}(\widehat{\mathcal{M}})} \lesssim_{p,p'}  \|x\|_{L^{p}(\mathcal{M} )}.
\end{eqnarray}
 \end{lem}
\begin{proof}We can always identify $L^1(\mathcal{M})$ with $\mathcal{M}_{*}$ via the map
$j(x)=\tau(x\cdot).$ Therefore, by formula (F3) in Definition \ref{Fourier-structure} and by Definition \ref{Fourier-transform},
we can define the linear map 
$$
T: x\mapsto \mathcal{F}[x], \quad x\in L^{1}(\mathcal{M})\cap L^{\infty}(\mathcal{M}),$$
which is well defined and extended to a bounded from $L^{1}(\mathcal{M})$ to $L^{\infty}(\widehat{\mathcal{M}})$ by (F1) and by the Plancherel identity (F2), we have    
\begin{equation*} 
\|\mathcal{F}[x]\|_{L^{2}(\widehat{\mathcal{M}})}=\|x\|_{L^{2}(\mathcal{M})},\quad x\in L^{2}(\mathcal{M}).
\end{equation*}
In other words, we obtain that $T$ is bounded from  $L^{1,1}(\mathcal{M})$ to $L^{\infty,\infty}(\widehat{\mathcal{M}})$ and bounded from $L^{2,2}(\mathcal{M})$ to $L^{2,2}(\widehat{\mathcal{M}}).$
 Hence, the assertion follows from the noncommutative Marcinkiewicz theorem (see \cite[Theorem 7.8.2, p. 434]{DPS}).  
\end{proof} 
By the second formula in (F1) and the Plancherel identity (F2) in Definition \ref{Fourier-structure}, we obtain the dualised version of the Hausdorff-Young inequality \eqref{R-H-Y_ineq}.
\begin{lem}\label{L.5.1} (Inverse Hausdorff-Young inequality). Let $1\leq{p}\leq2$  and $\frac{1}{p}+\frac{1}{p'}=1$. Then we have
\begin{eqnarray}\label{H-L_ineq}
\|x\|_{L^{p'}(\mathcal{M})}  \lesssim_{p,p'}   \|\mathcal{F}[x]\|_{L^{p,p'}(\widehat{\mathcal{M}})}.
\end{eqnarray}
\end{lem}

The following is a main result of the paper which shows the $L^p$-$L^q$ boundedness of the Fourier multipliers.
\begin{thm}\label{Hormander-thm} Let  $1<p\leq2\leq q<\infty$ and suppose that $\sigma\in L^{r,\infty}(\widehat{\mathcal{M}}),$ $\frac{1}{r}= \frac{1}{p} - \frac{1}{q}.$ Then the Fourier multiplier $A_\sigma$ is extended to a bounded map from $L^p(\mathcal{M})$ to  $L^q(\mathcal{M}),$ and we have
$$
\|A_\sigma\|_{L^p(\mathcal{M}) \rightarrow L^q(\mathcal{M})}\lesssim_{p,q}  \|\sigma\|_{L^{r,\infty}(\widehat{\mathcal{M}})}, \quad \frac{1}{r}= \frac{1}{p} - \frac{1}{q}.$$ 
\end{thm}
\begin{proof} Let $\frac{1}{r}:=\frac{1}{p} - \frac{1}{q}.$ Then we have that $\frac{1}{q'}=\frac{1}{r} +\frac{1}{p'}.$  Since $A_\sigma$ is a Fourier multiplier (see Definition \ref{Fourier-multp}), we have
$$\mathcal{F}[A_\sigma x]=\sigma\cdot\mathcal{F}[x], \quad x\in L^1(\mathcal{M})\cap L^{\infty}(\mathcal{M}).$$
Then for any $\sigma\in L^{r,\infty}(\widehat{\mathcal{M}})$ and
$x\in L^1(\mathcal{M})\cap L^{\infty}(\mathcal{M}),$ by Lemmas \ref{L.5.1} and \ref{L.5.2}, and Proposition \ref{prop-0.2}, we have
\begin{eqnarray*}\|A_\sigma x\|_{L^q(\mathcal{M})}&\overset{\eqref{H-L_ineq}}{\lesssim}&\|\mathcal{F}[A_\sigma x]\|_{L^{q',q}(\widehat{\mathcal{M}})}=\|\sigma \cdot \mathcal{F}[x]\|_{L^{q',q}(\widehat{\mathcal{M}})}\overset{\eqref{NC Lorentz space norm II}}{=}\|\mu(\sigma\cdot\mathcal{F}[x])\|_{L^{q',q}(\mathbb{R}_+)}\\
&\overset{\eqref{decreas-proper}}{\lesssim}&\|\mu(\sigma)\cdot\mu(\mathcal{F}[x])\|_{L^{q',q}(\mathbb{R}_+)}\overset{\eqref{Lorentz-embedding-2-oper}}{\lesssim}\|\mu(\sigma)\|_{L^{r,\infty}(\mathbb{R}_+)}\|\mu(\mathcal{F}[x])\|_{L^{p',q}(\mathbb{R}_+)}\\
&\overset{\eqref{NC Lorentz space norm II}}{=}&\|\sigma\|_{L^{r,\infty}(\widehat{\mathcal{M}})}\|\mathcal{F}[x]\|_{L^{p',q}(\widehat{\mathcal{M}})}\overset{\eqref{Lorentz-embedding-1-oper}}{\lesssim}\|\sigma\|_{L^{r,\infty}(\widehat{\mathcal{M}})}\|\mathcal{F}[x]\|_{L^{p',p}(\widehat{\mathcal{M}})}\\
&\overset{\eqref{R-H-Y_ineq}}{\lesssim}&\|\sigma\|_{L^{r,\infty}(\widehat{\mathcal{M}})}\|x\|_{L^{p}(\mathcal{M})}.
\end{eqnarray*}
In other words, we have 
\begin{eqnarray*}\|A_\sigma x\|_{L^q(\mathcal{M})}\lesssim\|\sigma\|_{L^{r,\infty}(\widehat{\mathcal{M}})}\|x\|_{L^{p}(\mathcal{M})},\quad x\in L^1(\mathcal{M})\cap L^{\infty}(\mathcal{M}).
\end{eqnarray*}
Now, since $L^1(\mathcal{M})\cap L^{\infty}(\mathcal{M})$ is dense in $L^p(\mathcal{M})$ (see \cite[Proposition 3.10.2 (i) and Section 3.10, pp. 202-205]{DPS}), we obtain the result for any $x\in L^p(\mathcal{M}),$
thereby completing the proof.
\end{proof}

 Next, we prove the so-called Paley type inequality.
 \begin{thm}\label{Paley-ineq}Let  $1<p\leq2$ and $\frac{1}{s}=\frac{2}{p}-1.$ Then for any $y\in L^{s,\infty}(\widehat{\mathcal{M}})$ and $x\in L^p(\mathcal{M})$ we have
$$
\|y\cdot\mathcal{F}[x]\|_{L^p(\widehat{\mathcal{M}})}\lesssim_{p,s}  \|y\|_{L^{s,\infty}(\widehat{\mathcal{M}})}\cdot\|x\|_{L^p(\mathcal{M})}.$$ 
\end{thm}
\begin{proof} By assumption we have $\frac{1}{p}=\frac{1}{p'} + \frac{1}{s}.$ Then for any $y\in L^{s,\infty}(\widehat{\mathcal{M}})$ and $x\in L^p(\mathcal{M})$ we obtain
\begin{eqnarray*}\|y\cdot\mathcal{F}[x]\|_{L^p(\widehat{\mathcal{M}})}&=&\|y\cdot\mathcal{F}[x]\|_{L^{p,p}(\widehat{\mathcal{M}})}\overset{\eqref{NC Lorentz space norm II}}{=}\|\mu(y\cdot\mathcal{F}[x])\|_{L^{p,p}(\mathbb{R}_+)}\\
&\overset{\eqref{decreas-proper}}{\lesssim}&\|\mu(y)\cdot\mu(\mathcal{F}[x])\|_{L^{p,p}(\mathbb{R}_+)}\overset{\eqref{Lorentz-embedding-2-oper}}{\lesssim}\|\mu(y)\|_{L^{s,\infty}(\mathbb{R}_+)}\|\mu(\mathcal{F}[x])\|_{L^{p',p}(\mathbb{R}_+)}\\
&\overset{\eqref{NC Lorentz space norm II}}{=}&\|y\|_{L^{s,\infty}(\widehat{\mathcal{M}})}\|\mathcal{F}[x]\|_{L^{p',p}(\widehat{\mathcal{M}})}\\
&\overset{\eqref{R-H-Y_ineq}}{\lesssim}&\|y\|_{L^{s,\infty}(\widehat{\mathcal{M}})}\|x\|_{L^{p}(\mathcal{M})}.
\end{eqnarray*}
This completes the proof.
\end{proof}
As consequences of Theorem \ref{Hormander-thm}, we obtain the main results in \cite[Theorem 5.1]{AR} and \cite[Theorem 1.3]{Z}.

\begin{cor}\label{cor-1}Let $G$ be a locally compact separable unimodular group and let $VN_{R}(G)$ be its von Neumann algebra generated by the right regular representations of $G.$ Let the assumptions of Theorem \ref{Hormander-thm} hold. Then we have 
$$\|A_\sigma\|_{L^p(G) \rightarrow L^q(G)}\lesssim_{p,q}  \|\sigma\|_{L^{r,\infty}(VN_{R}(G))}, \quad \frac{1}{r}= \frac{1}{p} - \frac{1}{q}.$$
\end{cor}
\begin{cor}\label{cor-2}Let $\left(\mathbb{G}, \Delta, \varphi_{L}, \varphi_{R}\right)$ be a locally compact quantum group with a coproduct map $\Delta$ and the left and right tracial Haar weights $\varphi_{L}$ and $\varphi_{R},$  respectively. Let $\left(\widehat{\mathbb{G}}, \widehat{\Delta}, \widehat{\varphi}_{L}, \widehat{\varphi}_{R}\right)$ be its dual quantum group with tracial states $\widehat{\varphi}_{L}$ and $\widehat{\varphi}_{R},$ respectively.
If $\mathcal{M}=L^{\infty}(\widehat{\mathbb{G}})$ and $\widehat{\mathcal{M}}=L^{\infty}(\mathbb{G})$ are the corresponding von Neumann
algebras and if the assumptions of Theorem \ref{Hormander-thm} are satisfied, then we have 
$$\|A_\sigma\|_{L^p(\mathcal{M}) \rightarrow L^q(\mathcal{M})}\lesssim_{p,q}  \|\sigma\|_{L^{r,\infty}(\mathcal{\widehat{M}})}, \quad \frac{1}{r}= \frac{1}{p} - \frac{1}{q}.$$
\end{cor}


\subsection*{Acknowledgment}
The authors would like to thank to Dr. Rauan Akylzhanov for his helpful discussions on Fourier structure in noncommutative spaces. The second author is partially supported by the grant No. AP23483532 of the Science Committee of the Ministry of Science and Higher Education of the Republic of Kazakhstan.

The authors were partially supported by Odysseus and Methusalem grants (01M01021 (BOF Methusalem) and 3G0H9418 (FWO Odysseus)) from Ghent Analysis and PDE center at Ghent University. The first author was also supported by the EPSRC grants EP/R003025/2 and EP/V005529/1, and FWO Senior Research
Grant G022821N.
Authors thank the anonymous referee for reading the paper and providing thoughtful comments, which improved the exposition of the paper.


\begin{thebibliography}{1}

\bibitem{AR} R.~Akylzhanov, M.~Ruzhansky, \textit{$L_p-L_q$ multipliers on locally compact groups}. J. Funct. Anal. \textbf{278}:3 (2020), 108324. 

\bibitem{AMR} R. Akylzhanov, S. Majid, and M. Ruzhansky, \textit{Smooth dense subalgebras and Fourier multipliers on compact quantum groups}, Comm. Math. Phys. \textbf{362}:3 (2018), 761--799.

\bibitem{ARN} R.~Akylzhanov, M.~Ruzhansky, E.D.~Nursultanov, \textit{Hardy-Littlewood, Hausdorff-Young-Paley inequalities, and $L_p-L_q$ Fourier multipliers on compact homogeneous manifolds}. J. Math. Anal. Appl. \textbf{479}:2 (2019), 1519--1548.

\bibitem{Cas} M.~Caspers, \textit{The $L_p$-Fourier transform on locally compact quantum groups}. J. Oper. Theory \textbf{69}:1 (2013), 161--193.

\bibitem{CK} M. Chatzakou, V. Kumar, \textit{$L_p-L_q$ Boundedness of Fourier Multipliers Associated with the Anharmonic Oscillator.} J. Fourier Anal. Appl. \textbf{29} 
(2023), 73.


\bibitem{Cow1993} M.~Cowling, S.~Giulini, S.~Meda, \textit{$L^p-L^q$ estimates for functions of the Laplace-Beltrami operator on noncompact symmetric spaces. I.} Duke Math. J. \textbf{72}:1 (1974), 109--150.


\bibitem{Cow1974} M.G. Cowling, \textit{Spaces $A_{qp}$ and $L^p-L^q$ Fourier Multipliers.} PhD thesis, The Flinders University of South Australia, 1974. 

\bibitem{DPS} P.G.~Dodds, B.~de Pagter and F.A.~Sukochev, \textit{Noncommutative Integration and Operator Theory}. Publisher Name Springer, Berlin, Heidelberg, 2023. 

\bibitem{FK} T.~Fack, H.~Kosaki, \textit{Generalized $s$-numbers of $\tau$-measurable operators}, Pacific J. Math., \textbf{123}:2 (1986), 269--300.

\bibitem{G2008} L.~Grafakos, \textit{Classical Fourier Analysis}. Second Edition. Springer-Verlag, New York, 2008. 


\bibitem{Hor} L.~H\"ormander, \textit{Estimates for translation invariant operators in $L_p$ spaces}. Acta Math., \textbf{104} (1960), 93--140.

\bibitem{KM} S.~Kaczmarz, J.~Marcinkiewicz, \textit{Sur les multiplicateurs des s\'{e}ries orthogonales}. Studia Math.\textbf{7} (1938), 73--81.

\bibitem{KV} J. Kustermans, S. Vaes, \textit{Locally compact quantum groups in the von Neumann algebraic setting.} Math. Scand. \textbf{92}:1 (2003), 68--92. 

\bibitem{K2} V. Kumar, M. Ruzhansky, \textit{$L_p-L_q$ Multipliers on commutative hypergroups.} J. Austral. Math. Soc. \textbf{115} (2023), 375--395.

\bibitem{K3}V. Kumar, M. Ruzhansky, \textit{$L_p-L_q$ boundedness of $(k,a)-$Fourier multipliers with applications to nonlinear equations.} Int. Math. Res. Not. IMRN, \textbf{2023}:2 (2023), 1073--1093.

\bibitem{K4} V. Kumar, M. Ruzhansky, \textit{Hardy-Littlewood inequality and Fourier multipliers on compact hypergroups.} J. Lie Theory \textbf{32}:2 (2022), 475--498.

\bibitem{LT16} D. Lenz, A. Teplyaev, \textit{Expansion in generalized eigenfunctions for Laplacians on graphs and metric measure spaces}. Trans. Amer. Math. Soc. \textbf{368}:7 (2016), 4933--4956.

 \bibitem{LSZ} S.~Lord, F.~Sukochev, D.~Zanin, \textit{Singular traces. Theory and applications.} De Gruyter Studies in Mathematics, {\bf 46}. De Gruyter, Berlin, 2013.
 

 \bibitem{Mc}E.~McDonald, \textit{Nonlinear partial differential equations on noncommutative Euclidean spaces.} J. Evol. Equ. \textbf{24}:16 (2024), 1--58.

\bibitem{Mikh} S.G.~Mikhlin, \textit{On the multipliers of Fourier integrals}. Dokl. Akad. Nauk SSSR (N.S.), \textbf{109} (1956), 701--703.


\bibitem{Mar} J. Marcinkiewicz, \textit{Sur les multiplicateurs des s\'{e}ries de fourier}. Studia Math. \textbf{8} (1939), 78--91.


\bibitem{M-Nur} M.~Nursultanov, \textit{$L_p-L_q$ boundedness of Fourier multipliers}. (2022), 1--22. arXiv:2210.11183

\bibitem{NT1} E.D.~Nursultanov,  N.T.~Tleukhanova, \textit{On multipliers of multiple Fourier series}. Tr. Mat. Inst. Steklova, \textbf{227} (Issled. po Teor. Differ. Funkts. Mnogikh Perem. i ee Prilozh. \textbf{18}) (1999), 237--242.

\bibitem{NT2} E.D.~Nursultanov,  N.T.~Tleukhanova, \textit{ Lower and upper bounds for the norm of multipliers of multiple trigonometric Fourier series in Lebesgue spaces}. Funktsional. Anal. i Prilozhen., \textbf{34}:2 (2000), 86--88.

\bibitem{PXu} G.~Pisier and Q.~Xu, \textit{Non-commutative $L_p$-spaces.}  In Handbook of the geometry of Banach spaces, North-Holland, Amsterdam, \textbf{2} (2003), 1459--1517. 

 \bibitem{RV} J.~Rozendaal,  M.~Veraar, \textit{Fourier multiplier theorems involving type and cotype.} J. Fourier Anal. Appl. \textbf{24}:2 (2018), 583--619.

 \bibitem{RST1} M.~Ruzhansky, S.~Shaimardan, K.~Tulenov, \textit{H\"ormander type Fourier multiplier theorem and Nikolskii inequality on quantum tori, and applications}. 	arXiv:2402.17353

 \bibitem{RST2} M.~Ruzhansky, S.~Shaimardan, K.~Tulenov, \textit{$L^p-L^q$ boundedness of Fourier multipliers on quantum Euclidean spaces}. 	arXiv:2312.00657

\bibitem{Stein} E.M.~Stein, \textit{Singular integrals and differentiability properties of functions.} Princeton Mathematical Series, \textbf{30}. Princeton University Press, Princeton, N.J., 1970.

\bibitem{Tim08} T. Timmermann, \textit{An invitation to quantum groups and duality}. EMS Textbooks in Mathematics. European Mathematical Society (EMS), Zürich, 2008. From Hopf algebras to multiplicative unitaries and beyond.

\bibitem{Z} H. Zhang, \textit{$L_p-L_q$ Fourier Multipliers on Locally Compact Quantum Groups.} J. Fourier Anal. Appl. \textbf{29} (2023), 46.
\end{thebibliography}
\end{document}